\documentclass[12pt,a4paper,oneside]{amsart}
\pdfoutput=1
\usepackage{color,amssymb,latexsym,amsfonts,textcomp,fancyhdr,calc,graphicx}
\usepackage{amsmath,amstext,amsthm,amssymb,amsxtra}
\usepackage{txfonts}
\usepackage[left=2.5cm,right=2.5cm,bottom=2cm,top=2cm]{geometry} %%%%%%%
\usepackage[colorlinks=true]{hyperref}
\usepackage{asymptote}

\usepackage{txfonts,pxfonts}%,tikz} %also txfonts 
\usepackage[T1]{fontenc}

\newtheorem{theorem}{Theorem}
\newtheorem{corollary}[theorem]{Corollary}

\newtheorem{conjecture}{Conjecture}
\newtheorem{proposition}{Proposition}
\newtheorem{definition}{Definition}[section]
\theoremstyle{definition}

\newcommand{\beql}[1]{\begin{equation}\label{#1}}
\newcommand{\eeq}{\end{equation}}
\newcommand{\comment}[1]{}

\newcommand{\Abs}[1]{{\left|{#1}\right|}}

\newcommand{\N}{{\bf N}}

\newcommand{\Set}[1]{{\left\{{#1}\right\}}}

\newcommand{\RR}{{\mathbb R}}

\newcommand{\CC}{{\mathbb C}}
\newcommand{\ZZ}{{\mathbb Z}}
\newcommand{\NN}{{\mathbb N}}

\newcounter{rem}
\newcounter{step}
\newcounter{mysec}
\newcounter{mysubsec}[mysec]
\newcounter{othm}
\def\theothm{\Alph{othm}}

\newcommand{\Tr}[1]{Theorem~\ref{#1}}

\newcommand{\Prr}[1]{Pro\-position~\ref{#1}}

\newcommand{\Cr}[1]{Corollary~\ref{#1}}

\newtheorem{problem}[conjecture]{{Problem}}

\newcommand{\eqco}{equilibrium configuration}

\begin{document}

\author{Agelos Georgakopoulos and Mihail N. Kolountzakis}
\address{M.K.: Department of Mathematics and Applied Mathematics, University of Crete, Voutes Campus, GR-700 13, Heraklion, Crete, Greece}
\email{kolount@gmail.com}
\address{A.G.: Mathematics Institute, University of Warwick, CV4 7AL, UK}
%\email{agelos.georgakopoulos@warwick.ac.uk} Δεν το θέλω 
\thanks{The first author is supported by the European Research Council (ERC) under the European Union's Horizon 2020 research and innovation programme (grant agreement No 639046).
The second author has been partially supported by the ``Aristeia II'' action (Project
FOURIERDIG) of the operational program Education and Lifelong Learning
and is co-funded by the European Social Fund and Greek national resources.}

\title{On particles in equilibrium on the real line}

\begin{abstract}
We study equilibrium configurations of infinitely many identical particles on the real line or finitely many particles
on the circle, such that the (repelling) force they exert on each other depends only on their distance.
The main question is whether each equilibrium configuration needs to be an arithmetic progression.
Under very broad assumptions on the force we show this for the particles on the circle. In the case of infinitely
many particles on the line we show the same result under the assumption that the maximal (or the minimal) gap between successive points is
finite (positive) and assumed at some pair of successive points.
Under the assumption of analyticity for the force field (e.g., the Coulomb force) we deduce some extra rigidity for the configuration: knowing
an equilibrium configuration of points in a half-line determines it throughout.
Various properties of the equlibrium configuration are proved.
\end{abstract}

\maketitle
\section{Introduction}

In this paper we study configurations of identical particles on the real line, or unit circle, 
that are in mechanical equilibrium when they exert repelling forces on each other that depend only on their distance.
We allow an arbitrary strictly monotone decreasing function $F: \RR_+ \to \RR_+$ to
determine the force between two particles as a function of their distance.

A folklore fact in the study of Wigner crystals is that for an infinite system of particles confined on the real line, or a finite system of particles confined on the unit circle ${\mathbb S}^1 \subset \RR^2$, and for various natural force fields, the only \emph{ground state}, i.e.\ the minimiser of the energy of the system, is obtained when the particles are equally spaced\footnote{\url{https://en.wikipedia.org/wiki/Wigner_crystal}}. Our first result is that this holds in much greater generality: for every strictly monotone decreasing function $F: \RR_+ \to \RR_+$, the only configurations of $n\in \N$ particles on ${\mathbb S}^1$ which are in mechanical equilibrium when the force between any two particles at distance $d$ is $F(d)$, are obtained when the distance between any two consequtive particles is constant. By \emph{(mechanical) equilibrium} we mean that the net force tangent to ${\mathbb S}^1$ exerted on each particle is zero (\Cr{corS1}). Similarly, we prove that the only periodic configurations of particles on $\RR$ in equilibrium are obtained by equally spacing the particles (\Cr{corper}). Even more, we prove that any configuration in equilibrium that attains the infimum or supremum of distances between consequtive particles is equally spaced.
 All these facts follow from a very simple argument (\Tr{extrem}), that, if new, might simplify the proofs of the aforementioned statement about ground states of specific potentials.

If the configuration is allowed to be aperiodic, then the problem is to the best of our knowledge open
even for specific force fields like e.g.\ a Coulomb force $F(d) = d^{-2}$.
In fact our original motivation was the following question asked by I. Benjamini \cite{RWatW}

\begin{problem}  \label{prb}
If a configuration of particles on $\RR$ is in (mechanical) equilibrium, do all distances between subsequent particles have to be equal?	
\end{problem}

\emph{Equilibrium} here means that the total force exerted on each particle from each side is finite, and the net force exerted on each particle is zero.

This problem is open for all strictly monotone decreasing functions $F: \RR_+ \to \RR_+$, and we find it interesting that, although it is not clear that the answer is  positive for e.g.\ $F(d) = d^{-2}$, it is also not clear whether there exists $F$ for which the answer is negative. We prove however that if we nail one of the particles at a fixed position, then we can obtain non-trivial \eqco s for continuous $F$ (\Tr{nailed}).

\medskip
We also obtain the following result about the Coulomb force (or somewhat more general analytic forces).
A configuration in equilibrium with bounded distances between consequtive particles is uniquely determined by any of its \emph{tails} (i.e.\ co-final subsequences); see \Tr{th:main}.

In the above discussion the particles are tacitly assumed to have equal masses. If we allow them to have different masses, then non-equally-spaced stable configurations do exist as observed by Ulam \cite[Chapter VII, \S 4]{Ulam}.

Stable particle configurations for generic force functions are also considered in \cite{cohn_universally_2007}, although with a different focus. For an analogue of \Prr{prb} in higher dimensions see \cite{cohn_point_2010}.

\section{No extremal gaps}

By an \emph{\eqco} we mean a bi-infinite sequence of real numbers such that a configuration of particles positioned at those numbers is in equilibrium in the sense defined above. An \eqco\ is \emph{trivial}, if it is an arithmetic progression, or in other words, if consequtive particles have equal distances.

For a pair of real numbers $x,y$, we write \emph{$xy$} for the absolute value of the force between a particle at $x$ and a particle at $y$. By a \emph{gap} we mean the distance (i.e.\ difference) between two consequtive members of an \eqco.

\begin{theorem} \label{extrem}
If an \eqco\ has a gap of maximal or minimal length, then it is trivial.
\end{theorem}	
\begin{proof}
	Suppose there is a non-trivial \eqco\ $\ldots, w_2, w_1, x, y, z_1, z_2, \ldots$, where the gap $[x,y]$ is maximal (see Fig.\ \ref{fig:max-length}), i.e.\ $|x-y|\geq |p-q|$ for any two (consequtive) members $p,q$ of the sequence. Since the \eqco\ is not trivial, we may assume without loss of generality that $|x-y| > |x-w_1|$. Writing $F^-(x)$ for the force exerted on a particle at $x$ from the left, we have

\begin{figure}[h]
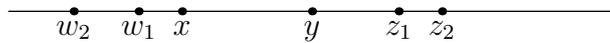

\begin{center}
\begin{asy}
import graph;
size(8cm);
pair l=(-4,0), r=(10,0);
pair w2=(-2.5, 0), w1=(-1,0), x=(0,0), y=(3,0), z1=(5,0), z2=(6,0);
draw(l -- r);
dot(w2); dot(w1); dot(x); dot(y); dot(z1); dot(z2);
label("$x$", x, S);
label("$y$", y, S);
label("$z_1$", z1, S);
label("$z_2$", z2, S);
label("$w_1$", w1, S);
label("$w_2$", w2, S);
\end{asy}
\end{center}
\caption{The points around a gap of maximum length $(x, y)$.} \label{fig:max-length}
\end{figure}

\begin{align} \label{sums}
\begin{split}
F^-(x) &= xw_1 + xw_2 + xw_3 + \ldots, \text{and} \\
F^-(y) &= yx + yw_1 + yw_2 + \ldots.
\end{split}
\end{align}

	Let us compare the $j$th summand of the first line to the $j$th summand of the second one: since $|x-y| > |x-w_1|$, we have $yx < xw_1$ by the strict monotonicity of the forces. Moreover, we have $|y -w_i| = |y-x| + |x- w_i| \geq |x - w_{i+1}| = |x - w_{i}|+ |w_{i} - w_{i+1}|$. Thus $yw_i \le xw_{i+1}$. Combining these inequalities we obtain $F^-(y)< F^-(x)$.
	
	By repeating the argument for the forces $F^+(y), F^+(x)$ exerted at $y,x$ from the right, the only difference being that $|y-z_1|$ might equal $|x-y|$,  we obtain $F^+(y)\geq F^+(x)$, reaching a contradiction.
	
	\medskip
	If the gap $[x,y]$ is minimal, then the same argument applies with all inequalities reversed.
\end{proof}

As an immediate corollary of \Tr{extrem}, we obtain

\begin{corollary} \label{corper}
	If an \eqco\ is periodic, then it is trivial.
\end{corollary} 

This can be adapted to configurations on the circle $S^1$:

\begin{corollary} \label{corS1}
Let $\{x_1, \ldots, x_n\}, x_i \in S^1$, be a configuration of particles constrained on $S^1$ in equilibrium. Suppose that the (tangential) force they exert on each other is a monotone decreasing function of their distance. Then the distance $d(x_i, x_{(i+1)mod n})$ of any two consequtive particles is constant.
\end{corollary} 
\begin{proof}
Suppose, to the contrary, that $d(x_i, x_{(i+1)mod n})$ is not constant. Then there are two consequtive particles $x,y$ maximising that distance, such that the distance between $x$ and its other neighbouring particle $w_1$ is strictly less that $d(x,y)$. We proceed as in the proof of \Tr{extrem}, the only difference being that now 
$F^-(x)$ denotes the force exerted on a particle at $x$ by particles lying on one of the two half circles $S_x^-$ between $x$ and its antipodal point $x'$ on $S^1$.

\begin{figure}[h]
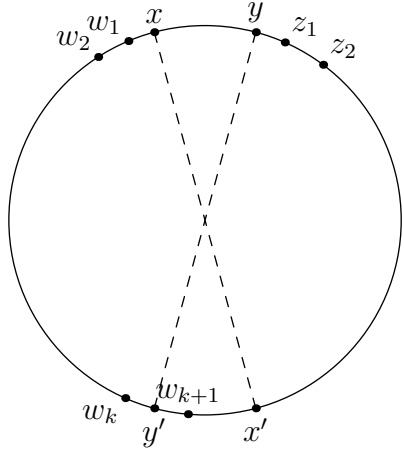

\begin{center}
\begin{asy}
import graph;
size(6cm);

draw(unitcircle);

pair up=(0,1);

pair x=rotate(15)*up, xx=-x, y=rotate(-15)*up, yy = -y;
pair w1 = rotate(8)*x, w2 = rotate(10)*w1;
pair z1 = rotate(-9)*y, z2 = rotate(-13)*z1;
pair wk = rotate(-9)*yy, wk1 = rotate(10)*yy;

dot(x); dot(xx); dot(y); dot(yy);
dot(w1); dot(w2); dot(wk); dot(wk1); dot(z1); dot(z2);
draw(x--xx, dashed); draw(y--yy, dashed);
label("$x$", x, N); label("$x'$", xx, S);
label("$y$", y, N); label("$y'$", yy, S);
label("$w_1$", w1, NW); label("$w_2$", w2, NW);
label("$z_1$", z1, NE); label("$z_2$", z2, NE);
label("$w_k$", wk, SW); label("$w_{k+1}$", wk1, N);
\end{asy}
\end{center}
\caption{The points around an arc of maximum length $(x, y)$.} \label{fig:max-length-circle}
\end{figure}

Thus the sums in \eqref{sums} have finitely many summands. Since the $i$th summand of the first sum is greater than the $i$th summand of the second one by the same argument, it suffices to show that the first sum has at least as many summands as the second. This is indeed true, for if $yw_k$ is the last summand of the second sum, then the particle $w_{(k+1)mod n}$ lies in $S_x^-$ because $d(w_k, w_{(k+1)mod n})\leq d(x,y)$ by the choice of $x,y$.

The rest of the proof is identical to that of \Tr{extrem}.
\end{proof}

\section{Uniqueness of continuation under analytic forces} \label{secuniq}
\begin{definition}
Call an increasing sequence $x_n \in \RR$, $n\in\ZZ$, {\em uniformly discrete}
if there are constants $0<c\le C<\infty$ such that
$$
c \le x_n-x_{n-1} \le C,\ \ \ \forall n \in \ZZ.
$$
\end{definition}

We show the following.

\begin{theorem}\label{th:main}
Let $x_n \in \RR$, $n\in\ZZ$ be a uniformly discrete configuration of particles
subject to repellent Coulomb forces
$$
F(d) = \frac{1}{d^2}.
$$
Suppose that the particles at the set $\Set{x_n \ge 0}$ are in equilibrium.
Then the locations $\Set{x_n < 0}$ are uniquely determined.
\end{theorem}

\begin{proof}
Suppose not, and suppose that the two sets of points $X, Y \subseteq (-\infty,0)$
(each of them uniformly discrete, in the obvious way)
can both cause the
electrons at the points $W = \Set{x_n \ge 0}$ to experience zero total force.
In other words, the two systems of electrons, at $X \cup W$ and at $Y \cup W$ are such that the electrons at $W$
are in equilibrium.
It follows that for each $w \in W$ the force exerted on $w$ due to electrons at $X$ is the same as the force exerted on $w$ due
to electrons at $Y$.

The Coulomb force exerted at a point $w$ on the nonnegative real semi-axis by the electrons at $X$ is given by
$$
f_X(w) = \sum_{x \in X} \frac{1}{(x-w)^2},
$$
up to redefining the physical constants, and similarly for the force $f_Y(z)$ due to electrons in $Y$.
Since these must be the same at each $w \in W$ we deduce that the function
\beql{fnc}
f(w) = f_X(w) - f_Y(w) = \sum_{p \in X \triangle Y} \frac{\epsilon_p}{(p-w)^2},
\eeq
where $X \triangle Y$ is the symmetric difference of $X$ and $Y$ and $\epsilon_p=\pm 1$ depending on whether $p \in X$ or $p \in Y$,
vanishes at each $w \in W$.
It is easy to see that $f(w)$ is well defined (the series at \eqref{fnc} converges) at every point of the complex plane except at $X \triangle Y$, at each point of which it has a pole of order 2,
and is an analytic function in $\CC\setminus(X\triangle Y)$.
Since, for $\Re w\ge 0$ we have
$$
\Abs{f(w)} \le \sum_{p \in X \triangle Y} \frac{1}{\Abs{p-w}^2} \le \sum_{p \in X \triangle Y} \frac{1}{\Abs{p}^2} < \infty,
$$
it is clear that $f$ is bounded on the closed right half plane. Our plan is to use \Tr{th:rudin} below to show that $f$ is identically 0.

We write, as we may,
$$
W = \Set{w_0=0 < w_1 < w_2, \ldots}
$$
for the points of $W$ and we assume that $c\le w_n - w_{n-1} \le C$ for all $n>0$. This implies that
\beql{growth}
c n \le w_n \le C n,\ \ \ (n \ge 0).
\eeq
Define the linear fractional transformation
$$
z = z(w) = \frac{w-1}{w+1},\ \ \ w = w(z) = \frac{1+z}{1-z}
$$
and note that $z(w)$ maps the open right half plane $\Set{\Re{w} > 0}$ bijectively to the open unit disk $\Set{\Abs{z}<1}$ (with $1 \to 0, 0 \to -1, i \to i$).

The function $f(w)$ vanishes at all points of $W$ and therefore the analytic function on the unit disk $\Set{\Abs{z}<1}$
$$
g(z) = f(w(z))
$$
vanishes at all (real) points $z_n = z(w_n) = 1-\frac{2}{w_n+1}$, $n\ge 0$, of the open unit disk.
Since $f$ is bounded on the open right half plane so is $g$ on the open unit disk.

Because of \eqref{growth} we have
\beql{roots}
1-z_n = \frac{2}{1+w_n} \ge \frac{2}{1+Cn}
\eeq
and hence
\beql{diverge}
\sum_n(1-\Abs{z_n}) = \infty.
\eeq
We now use the following result.
\begin{theorem}[\cite{rudin}, Theorem 15.23]\label{th:rudin}
If a function $g$ is analytic and bounded in the open unit disk $U$ and vanishes at points $z_n \in U$ 
satisfying \eqref{diverge}, then $g$ is identically 0 in $U$.
\end{theorem}
(This is a rather simple consequence of Jensen's formula.)

Thus Theorem \ref{th:rudin} implies that $g \equiv 0$ on $U$, hence $f \equiv 0$ on the open right half plane,
and by analytic continuation $f$ is $0$ on $\CC\setminus(X \triangle Y)$. So $f$ has no singularities at all, a contradiction,
unless $X = Y$, as we had to prove.
\end{proof}

\begin{corollary}\label{periodic-tail}
	Let $S$ be a uniformly discrete \eqco\ such that some tail of $S$ is periodic. Then $S$ is trivial.
\end{corollary}
\begin{proof}
	Let $T$ be such a tail, and let $T'$ be the subsequence of $T$ obtained by omitting the first period. By \Tr{th:main}, $T$ can be brought to equilibrium by a unique sequence preceeding it. We claim that this sequence must start with the period of $T$. Indeed, applying \Tr{th:main} to $T'$, and noting that $T$ is a shifted copy of $T'$, we see that the two unique continuations coincide.
	
	This easily implies that the whole sequence $S$ is periodic, and by \Cr{corper} trivial.
\end{proof}

\noindent{\bf Generalization.}
The proof of Theorem \ref{th:main} and Corollary \ref{periodic-tail} is valid for more general forces than the Coulomb forces.
The force function $F(d)$ needs to be an analytic function on the open right half complex plane, whose values on the positive real axis
are positive and satisfies
$$
\int_1^{+\infty} F(x)\,dx < \infty.
$$
Such functions are, for instance, the functions
$$
F(d) = \frac{1}{d^k},\ \ \ (k\ge 2),
$$
and
$$
F(d) = e^{-d^k},\ \ \ (k \ge 1).
$$

\section{Other remarks}

The following facts are easy to check
\begin{proposition} \label{ratios}
	If $x,y,z$ are consequtive points in an \eqco, then\\ $|x-y|/|y-z|$ is bounded.
\end{proposition}

Proposition~\ref{ratios} is an easy consequence of

\begin{proposition} \label{monoforce}
	If $S$ is a finite set of consequtive particles in an \eqco, then the (signed) forces exerted on $S$ by particles in $S$ only are monotone.
\end{proposition}
\begin{proof}[Proof (Sketch)]
	If they are not, then the other forces only make the situation worse.
\end{proof}

\begin{figure}[h]
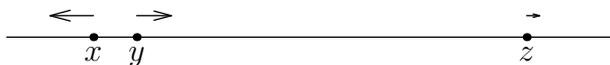

\begin{center}
\begin{asy}
import graph;
size(8cm);
pair up=(0,0.5);
pair l=(-2,0), r=(12,0);
pair x=(0,0), y=(1,0), z=(10,0);
draw(l -- r);
dot(x); dot(y); dot(z);
label("$x$", x, S);
label("$y$", y, S);
label("$z$", z, S);
draw(x+up--x+up-(1, 0), Arrow(SimpleHead));
draw(y+up--y+up+(0.8, 0), Arrow(SimpleHead));
draw(z+up--z+up+(0.3, 0), Arrow(SimpleHead));
\end{asy}
\end{center}
\caption{Why Proposition \ref{monoforce} implies Proposition \ref{ratios}.} \label{fig:three}
\end{figure}

To see why Proposition \ref{monoforce} implies Proposition \ref{ratios} let $x<y<z$ be three consecutive points in
an equilibrium configuration, hold the points $x, y$ fixed and let $z$ move far to the right (see Fig.\ \ref{fig:three}).
Observing the inner forces of the triple we see that if $z$ is far enough to the right then the force on $x$ is negative
(it is mostly affected by $y$), the force on $y$ is positive (it is mostly affected by $x$) and the force on $z$ is positive but
very small, violating the monotonicity proved in Proposition \ref{monoforce}.

In fact, the particles in Proposition~\ref{monoforce} do not have to be part of an \eqco; the statement holds for any finite set $S$ of consequtive particles that are in equilibrium inside some configuration. Even stronger, the first particle in $S$ does not have to be in equilibrium.

\begin{proposition} \label{extension}
Suppose that the force function $F$ is strictly monotone decreasing and continuous, and $\int_{1}^\infty F(x) dx< \infty$. For every uniformly discrete sequence of particles $S_- = \{x_{-1}, x_{-2}, \ldots\}$ with $x_i < x_{i-1} < 0$, there is a sequence of particles $S_+ =  \{x_0, x_{1}, x_{2}, \ldots\}, x_i > x_{i-1} > 0$, such that each particle in $S_+$ is in equilibrium in the configuration $S_- \cup S_+ = \{x_i\}_{i\in \ZZ}$. Moreover, $x_0$ can be chosen arbitrarily.
\end{proposition}
\begin{proof}
For $n= 1,2,\ldots$, let $x_n$ be any positive real. Then there are $x^n_1, \ldots x^n_{n-1} \in (0,x_n)$ such that the particles at $\{x^n_1, \ldots x^n_{n-1}\}$ are in equilibrium in the configuration\\ $S_- \cup \{x^n_1, \ldots x^n_{n-1}\} \cup \{x_n\}$: we claim that the positions in $(0,x_n)$ minimising the energy of the particles at $\{x^n_1, \ldots x^n_{n-1}\}$ have this property. 
To make this argument precise, define the energy $E(x,y)$ contributed by a pair of particles at positions $x$ and $y$ by $E(x,y):= \int_{z=|x-y|}^\infty F(z) dz$. Note that this is finite by the choice of $F$ and the fact that $S_-$ is uniformly discrete. 

For $m\in \NN_{>0}$, let $S_-^m$ be the subsequence $ \{x_{-1}, \ldots, x_{-m}\}$ of $S_-$. In order to obtain the desired configuration $\{x^n_1, \ldots x^n_{n-1}\}$ we will consider a sequence of configurations $C^m= x^{n,m}_1, \ldots x^{n,m}_{n-1}$, where $x^{n,m}_i \in (0,x_n)$, such that the particles in $C^m$ are in  equilibrium in the configuration $S_-^m \cup C^m \cup \{x_n\}$ and use compactness to take a limit.

For this, given $m$ we define the energy $E= E(x^n_1, \ldots x^n_{n-1})$ of the configuration\\ $S_-^m \cup \{x^n_1, \ldots x^n_{n-1}\} \cup \{x_n\}$ to be 
$$E:=\sum_{i=0}^{n-1} \sum_{j=-1}^{-m} E(x^{n}_i,x_j) + \sum_{i=0}^{n-1} \sum_{j=0}^{n-1} E(x^{n}_i,x^{n}_j),$$
i.e.\ the energy contributed by all pairs involving at least one of the particles $\{x^n_1, \ldots x^n_{n-1}\}$.  It is not hard to see that there are $x^{n,m}_1, \ldots x^{n,m}_{n-1} \in (0,x_n)$ minimising $E$ by the continuity of $F$ and the fact that $E$ increases if a particle gets too close to $x_{-1}$ or $x_n$. Note that the partial derivative of $E$ with respect to $x^n_i$ equals the total force exerted on the particle at $x^n_i$ by the definition of $E$, and by Fermat's theorem this has to vanish at any configuration minimising $E$. Thus each particle in $\{x^{n,m}_1, \ldots x^{n,m}_{n-1}\}$  is in equilibrium as claimed. By a standard compactness argument, there is a sequence $m_1, m_2, \ldots$ such that the position of $x^{n,m_j}_i$ converges, for each $i$, as $m_j$ goes to infinity. Define the limit configuration by $x^n_i:= \lim x^{n,m_j}_i$. It now follows easily from the continuity of $F$ that the particles at $\{x^n_1, \ldots x^n_{n-1}\}$ are in equilibrium in the configuration $S_- \cup \{x^n_1, \ldots x^n_{n-1}\} \cup \{x_n\}$ as desired.

Moreover, by the monotonicity and continuity of the forces, choosing $x_n$ appropriately we can ensure that $x_0$ equals any predetermined constant greater than $x_{-1}$.

By a compactness argument like the one used above, there is a sequence $n_1, n_2, \ldots$ such that the position of $x^{n_j}_i$ converges (possibly to infinity), for each $i$, as $n_j$ goes to infinity. By Proposition~\ref{ratios} (see also the remark after Proposition~\ref{monoforce}), the limit of $x^{n_j}_i$ is finite. Then defining $S_+ = \{\lim_j x^{n_j}_i\}_{i\in \ZZ}$ satisfies our requirements (here, we use the continuity of the forces again).
\end{proof}

\begin{proposition} \label{extension2}
If the gaps of $S_-$ in Proposition~\ref{extension} are bounded between real numbers $0<b<B$, then $S_+$ can be chosen so that its gaps are bounded between $\min(b, x_0-x_{-1})$ and $\max(B, x_0-x_{-1})$.
\end{proposition}
\begin{proof}
We repeat the proof of Proposition~\ref{extension}, except that we replace the particle at $x_n$ with an 1-way infinite arithmetic progression $x_n, x_n +a , x_n +2a, \ldots$, where $a$ is any real in $(b,B)$. We claim that, for every $n\in \NN$, the resulting gaps of $\{x^n_1, \ldots x^n_{n-1}\}$ are bounded between $\min(b, x_0-x_{-1})$ and $\max(B, x_0-x_{-1})$. Indeed, let $x,y$ be the particles spanning the largest (respectively smallest) gap of $\{x^n_1, \ldots x^n_{n-1}\}$. If this gap is longer than $\max(B, x_0-x_{-1})$ (resp.\ smaller than $\min(b, x_0-x_{-1})$), then we can repeat the main argument of the proof of Theorem~\ref{extrem} to obtain a contradiction, since such a gap cannot involve any particle in $S_-$, and in that proof we only used the equilibrium for the particles $x,y$.
\end{proof}

We remark that we do not know if Proposition~\ref{extension2} is true for every \eqco\ $S_+$. (We only proved it for one \eqco.)

\medskip
Finally, we adapt the proof of Proposition~\ref{extension} to obtain the main result of this section

\begin{theorem}  \label{nailed}
For every strictly monotone decreasing and continuous force function $F: \RR_+ \to \RR_+ $, there is a configuration $ \{\ldots, x_{-2}, x_{-1}, x_0, x_{1}, x_{2},\ldots\} $ of particles on $\RR$ in which all particles except $x_0=0$ are in equilibrium, and $x_{-1}$ and $x_{1}$ can be chosen arbitrarily. %(they do not have to be equal in absolute value).
\end{theorem} 
\begin{proof}
We use the strategy of the proof of Proposition~\ref{extension}, except that we replace the sequence $S_-$ with a single particle at a position $x_{-n}$, we fix a particle at $x_0=0$ which does not have to be in equilibrium, and we introduce particles $\{x^n_{-1}, \ldots x^n_{-(n-1)}\}$ in equilibrium for each $n\in \NN$. We need to show that, by choosing $x_{-n}, x_n$ appropriately, we can bring the particles $x^n_{-1}, x^n_1$ to the desired positions for each $n$. We can then take a limit of such  configurations as $n\to \infty$ as in Proposition~\ref{extension}.

Define a \emph{0-centered configuration} to be a sequence $ \{x_{-n}, \ldots, x_{-1}, x_0=0, x_{1}, \ldots\, x_n\} $  in which all particles except possibly $x_{-n}, x_0, x_n$ are in equilibrium (when forces between particles are given by $F$).
Thus it remains to prove that for every $a<0, b>0$, $n\in \NN_*$, there is a 0-centered configuration $ \{x_{-n}, \ldots, x_{-1}, x_0, x_{1}, \ldots\, x_n\} $ with $x_{-1}=a$ and $x_1=b$.

To prove this, let 
$$d:= \sup \{x - x' \mid \text{ there is a 0-centered configuration with } x_{-n} =x', x_n=x, x_{-1} \geq a, \text{ and } x_1 \leq b \}.$$

Let us prove that $d< \infty$. Let $\{x_{-n}, \ldots, x_{-1}, x_0, x_{1}, \ldots\, x_n\} $ be a candidate configuration. Since $x_1 \leq b$, the force to the right exerted on $x_1$ from particle $0$ is at least $F(b)$, and has to be balanced by the particles $x_2, \ldots, x_n$. This gives an upper bound $b'$ on $x_2$, as the force on $x_1$ to the left is less than $(n-1) F(x_2-x_1)$ by the monotonicity of $F$. Similarly, the force to the right exerted on $x_2$ from particle $0$ is lower bounded by $F(b')$, and this imposes an upper bound on $x_3$, and so on up to $x_n$. Applying the same argument to the negative particles we also see that $x_{-n}$ is bounded, hence $x_n-x_{-n}$ is bounded.

Since $F$ is continuous, this supremum is attained by some 0-centered configuration $Y= \{y_{-n}, \ldots, y_{-1}, y_0, y_{1}, \ldots\, y_n\} $. We claim  
that $y_{1}= a$ and $y_1 =b$ in $C$, which would complete our proof. 

Suppose to the contrary that  $y_{-1} = a + \epsilon$ for some $\epsilon>0$ (and possibly $y_1<b$). We will produce a 0-centered configuration $Y'= \{y'_{-n}, \ldots, y'_{-1}, y_0=0, y'_{1}, \ldots\, y'_n\} $ where $y'_{-n} = y_{-n}- \epsilon$ and $y'_i \in [y_i-\epsilon, y_i]$ for every $i$, and in fact $y'_0= y_0=0$ and $y'_n=y_n$. This contradicts the choice of $Y$ as $Y'$ increases $d$ by $\epsilon$, and satisfies all other requirements. 

We will obtain $Y'$ as a limit of sequences $Y^j= \{y^j_{-n}, \ldots, y^j_{-1}, y_0=0, y^j_{1}, \ldots\, y^j_n\}, j= 0,1,\ldots$. 

To begin with, we define $Y^0$ by letting $y^0_{-n} = y_{-n}- \epsilon$, and letting $y^0_i= y_i$ for every other $i$. In fact, we will never change the position of the particle $-n$ again, that is, we fix $y^j_{-n} = y_{-n}- \epsilon$ for every $j$. We will also never change the positions of particles $0$ and $n$; we call the particles $-n, 0, n$  the \emph{fixed particles}.

Note that no non-fixed  particle is in equilibrium in $Y^0$: for all non-fixed particles, we have reduced the force from the left in comparison to $Y$, and kept the force from the right fixed. By the continuity and monotonicity of $F$, there is a position $y\in (y^0_{-n}, y^0_{-n+1})$ such that if we move the particle $-n+1$ from  $y^0_{-n+1}$ to $y$, then that particle will be in equilibrium. We now define $Y^1$ by letting $y^1_{-n+1}=y$ and $y^1_i= y^0_i$ for every other $i$.

We proceed similarly with the next particle $-n+2$. Since we moved the previous two particles to the left, it is still true that the net force on that particle from the left has been reduced, and we move it to the left to a position $y^2_{-n+2}$ to bring it to equilibrium and define $Y^2$ (we are aware that the particle $-n+1$ is not any more in equilibrium in $Y^2$).

We proceed inductively to define the sequences $Y^3,Y^4, \ldots, Y^{2(n-1)}$, each of which only moves the position of the non-fixed particle $-n+3, \ldots, -1, 1, \ldots n-1$ respectively to the left. Note that after these changes, all non-fixed particles but $n-1$ are again out of equilibrium, and the net force they experience is to the left. We repeat another round of similar changes, obtaining sequences $Y^{2(n-1)+1}, \ldots, Y^{4(n-1)}$ in which particle $-n+1, \ldots, -1, 1, \ldots n-1$ respectively are moved to the left to reach  a temporary equilibrium. After we are done we perform another such round, and so on ad infinitum.

Since each $y^j_i$ is monotone decreasing in $j$, and bounded below by $y^0_{-n}$, it converges to some value $y'_i$, and we use these values to define the limit configuration $Y'$.

Next, we claim that for every $j$, and every particle $i$, we have $y^j_i\in [y_i-\epsilon, y_i]$. For if not, then consider the first step $j$ when a counterexample $y^j_i$ arises. Then particle $i$ has to be in equilibrium in $Y^j$ because it must have just been moved. Let us compare the forces exerted on this particle in $Y^j$ to those exerted on it in $Y$. All particles have been moved to the left if at all, and particle $i$ has experienced the largest displacement as all other particles have moved by at most $\epsilon$. But this means that all particles to the left of $i$ are closer to $i$ in $Y^j$  than they were in $Y$, and all particles to the right of $i$ are further from $i$ in $Y^j$  than they were in $Y$. By the strict monotonicity of $F$, this contradicts the fact that $i$ was in equilibrium in both $Y$ and $Y^j$.

This proves our claim, which implies that $y'_i\in [y_i-\epsilon, y_i]$ for every $i$. In particular, $y'_{-1}\geq y_{-1}-\epsilon$, and so $y'_{-1}\geq a$ (and clearly $y'_1 \leq b$). Thus $Y'$ is a candidate for the definition of $d$ if it is 0-centered. And indeed it is: since the positions of all particles converge in the sequence $(Y^j)$, the net force on each particle $i$ converges as $j\to \infty$ by the continuity of $F$, and as it assumes the value 0 infinitely often ---namely, at steps at which we bring $i$ to equilibrium--- it has to converge to 0.

Thus $Y'$ contradicts the choice of $Y$ as claimed, which proves that $y_{-1} = a$. By the same arguments, we obtain a contradiction if $y_1<b$ by moving all particles to the right a bit.
\end{proof}

\section{Questions}

\Tr{th:main} says that if a 1-way infinite sequence of particles $S_+$ (at bounded distances) can be brought to equilibrium by another 1-way infinite sequence $S_-$, then $S_+$ uniquely determines $S_-$. We can ask if the converse can be proved: if $S_-$ can be used to bring some $S_+$  to equilibrium, is $S_+$ uniquely determined by $S_-$?

\section*{Acknowledgement}
We would like to thank Florian Theil for various suggestions.

\end{document}